\documentclass[12pt]{amsart}
\usepackage{amsmath}
\usepackage{amsfonts}
\usepackage{latexsym}
\usepackage{amssymb}
\usepackage{enumerate}
\usepackage[dvips]{graphics}

\newcommand{\Z}{{\mathbf Z}}
\newcommand{\Q}{{\mathbf Q}}
\newcommand{\R}{{\mathbf R}}

\newcommand{\coker}{{\mbox{\rm coker}}}

\newcommand{\supp}{\rm {supp}}
\newcommand{\rk}{{\rm {rk}}}
\newcommand{\comment}[1]{}

\def\ker{{\rm Ker }}

\def\R{\mathbf R}

\newtheorem{theorem}{Theorem}[section]

\newtheorem{proposition}[theorem]{Proposition}
\newtheorem{lemma}[theorem]{Lemma}
\newtheorem{example}{Example}
\newtheorem{corollary}[theorem]{Corollary}
\newtheorem{definition}{Definition}
\newtheorem{remark}[theorem]{Remark}

\begin{document}
\title[Two particles on a graph]{Topology of configuration space of two particles on a graph, I}
\author[Kathryn ~Barnett and Michael ~Farber ]{Kathryn ~Barnett and Michael~Farber}
\address{Department of Mathematics, University of Durham, Durham DH1 3LE, UK}
\email{Kathryn.Barnett@durham.ac.uk}
\email{Michael.Farber@durham.ac.uk}



\date{\today}


\begin{abstract}
In this paper we study the homology and cohomology of configuration spaces $F(\Gamma, 2)$ of two distinct particles on a graph $\Gamma$. Our main tool is  intersection theory for cycles in graphs. We obtain an explicit description of the cohomology algebra $H^\ast(F(\Gamma, 2))$ in the case of planar graphs.
\end{abstract}

\maketitle

Let $F(X, n)$ denote the space of configurations of $n$ distinct points lying in a topological space $X$, i.e.
$$F(X, n)=\{(x_1, x_2,\dots,x_n)\in X\times   \dots \times X; x_i \not= x_j \, \, \mbox{for}\, \, \, i\not=j\}.$$
Spaces $F(X, n)$, first introduced by Fadell and Neuwirth in \cite{FN}, play an important role in modern topology and its applications. Topology of configuration spaces $F(X, n)$ is well-studied and many important results have been obtained, see for example \cite{Ar}, \cite{Co}, \cite{Va}, \cite{To}. The best understood case is when $X=\R^m$ is a Euclidean space; the cohomology algebra of $F(\R^m,n)$ is described by the theory of subspace arrangements. The Totaro spectral sequence \cite{To} allows one to compute the cohomology algebra of $F(X, n)$ when $X$ is a smooth manifold.

In this paper we study spaces $F(\Gamma, 2)$ when $\Gamma$ is a finite graph; these spaces appear in topological robotics as configuration spaces of two objects moving along a one-dimensional network without collisions, see \cite{Gr}, \cite{GK}, \cite{Far}, \cite{Far1}. The space $F(X, 2)=X\times X - \Delta_X$ is also known under the name of \lq\lq deleted product\rq\rq; deleted products of graphs were studied in \cite{Patty61}, \cite{Patty62}, \cite{Co} and \cite{CP}.

Unfortunately several published papers about the topology of $F(\Gamma, 2)$ contain serious errors. For example, Theorem 4.2 from \cite{Patty62} is incorrect, and paper \cite{Cop}, page 1006, gives a wrong description of the second homology group of $F(\Gamma, 2)$. Regretfully these mistakes were not explicitly acknowledged and analyzed in the subsequent work of H. Copeland and C.W. Patty. However they were mentioned implicitly; thus, in the abstract to \cite{CP} the authors write: {\it \lq\lq the two dimensional Betti numbers of $F(\Gamma, 2)$ are larger than they were originally thought to be\rq\rq}.

Recently, important progress in the analysis of the topology of configuration spaces of graphs was made in the work of A. Abrams \cite{Ab} and D. Farley and L. Sabalka \cite{FS1}, \cite{FS2}, \cite{FS3}, \cite{FS4}, \cite{Farley}. As a result, cohomology algebras of unordered configuration spaces of trees were computed; the case of two point configuration spaces of trees was studied in \cite{Far}.

In this paper we describe an intersection theory for cycles in graphs which is crucial for the study of Betti numbers of configuration spaces $F(\Gamma, 2)$.
This theory allows us to find explicit bases for $H_i(F(\Gamma, 2))$ where $i=1, 2$, for planar graphs $\Gamma$. In the final section we describe the cup-product $\cup: H^1(F(\Gamma, 2)) \times H^1(F(\Gamma, 2)) \to H^2(F(\Gamma, 2))$.
To illustrate our results we state the following theorem (see Theorem \ref{thm3}):

{\bf Theorem.} {\it Let $\Gamma\subset \R^2$ be a connected planar graph such that every vertex $v$ has valence $\mu(v)\ge 3$. Denote by $U_0$, $U_1$,  $\dots, U_r$ the connected components of the
 complement $\R^2-\Gamma$ where $r=b_1(\Gamma)$ and $U_0$ is the unbounded component. Assume that (1) the closure of every domain $\bar U_i$ with $i= 1, \dots, r$ is contractible, and $\bar U_0$ is homotopy equivalent to the circle $S^1$ and (2) for every pair $i, j\in \{0, 1, \dots, r\}$
the intersection $\bar U_i \cap \bar U_j$ is connected.
Then the Betti numbers of $F(\Gamma, 2)$ are given by
\begin{eqnarray}
b_1(F(\Gamma, 2)) = 2b_1(\Gamma) + 1
\end{eqnarray}
and
\begin{eqnarray}
\qquad  b_2(F(\Gamma, 2)) = b_1(\Gamma)^ 2- b_1(\Gamma) +2 - \sum_{v\in V(\Gamma)}(\mu(v)-1)(\mu(v)-2).
\end{eqnarray}
Here $V(\Gamma)$ denotes the set of vertices of $\Gamma$.}

We also describe explicit generators of $H_i(F(\Gamma, 2); \Q)$ for $i=1, 2$.

In this paper the symbols $H_\ast(X)$ and $H^\ast(X)$ denote homology and cohomology groups with integral coefficients. The other coefficient groups in homology and cohomology are indicated explicitly.

\section{Basic facts about $F(\Gamma, 2)$}

Let $\Gamma$ be a finite graph, i.e. a finite simplicial complex of dimension one. As usual, {\it edges} of $\Gamma$ are defined as
closures of 1-dimensional simplices.

For a point $x\in \Gamma$ its {\it support} $\supp\{x\}$ is defined as the closure of the simplex containing $x$. In other words, if $x$ is a vertex of $\Gamma$ then $\supp\{x\}=x$ and if $x$ lies in the interior of an edge $e$ then $\supp\{x\}= e$.

Denote by $D(\Gamma, 2)\subset \Gamma\times \Gamma$ the set of all pairs $(x, y)\in \Gamma\times \Gamma$ such that $\supp\{x\}$ and $\supp\{y\}$ are disjoint. Clearly, $D(\Gamma, 2)$ is a closed subset of $\Gamma\times \Gamma$ and $D(\Gamma, 2)$ is contained in $F(\Gamma, 2)$. Moreover,
$D(\Gamma, 2)$ is a subcomplex of $\Gamma\times \Gamma$, viewed with its obvious cell-complex structure. The cells of $D(\Gamma, 2)$ are as follows:
(0) zero-dimensional cells are ordered pairs $uv$ where $u$ and $v$ are distinct vertices of $\Gamma$; (1) one-dimensional cells are of two types $ev$ and $ve$ where $e$ is an edge and $v$ is a vertex not incident to $e$; (2) two-dimensional cells of $D(\Gamma, 2)$ have the form $ee'$ where $e$ and $e'$ are edges of $\Gamma$ having no common vertices. To explain our notations, note that $ee'$ is the set of all configurations $(x,y)$ with $x\in  e$ and $y\in e'$.

Consider the involution $\tau: \Gamma\times \Gamma \to \Gamma\times \Gamma$ permuting the points, i.e. $\tau(x,y)=(y,x)$ for $x,y\in \Gamma$. Clearly $\tau$ induces involutions on $F(\Gamma, 2)$ and on $D(\Gamma, 2)$.

\begin{lemma}\label{defret} There exists an equivariant strong homotopy retraction $F(\Gamma, 2)\to D(\Gamma, 2)$.
\end{lemma}

More precisely, we claim that there exists a continuous homotopy $h_t: F(\Gamma, 2)\to F(\Gamma, 2)$ where
$t\in [0,1]$,
 with the properties $h_t\tau =\tau h_t$, $h_0={\rm {id}}$, $h_t|D(\Gamma, 2)={\rm {id}}$, and $h_1(F(\Gamma, 2))=D(\Gamma, 2)$.

This result is well-known, see A. Shapiro \cite{Shapiro}, and W.- T. Wu \cite{Wu}, \cite{Wubook}. Note that the proof of Lemma 2.1 from \cite{Shapiro} is incorrect. Instead we refer the reader to the argument of the proof of Theorem 2.4 from \cite{Ab} which gives an equivariant deformation retraction of $F(\Gamma, 2)$ onto $D(\Gamma, 2)$, as required.

In view of Lemma \ref{defret} we may replace $F(\Gamma, 2)$ by $D(\Gamma, 2)$ while studying homotopy properties of $F(\Gamma, 2)$. The space $D(\Gamma, 2)$ has the advantage of being a finite polyhedron.

In this paper we discuss the homology of the configuration spaces of graphs. In connection with this the following statement is useful:

\begin{corollary}\label{euler} Let $V(\Gamma)$ denote the set of vertices of $\Gamma$ and $\mu(v)$ be the number of edges incident to a vertex $v\in V(\Gamma)$. Then
the Euler characteristic $ \chi(F(\Gamma, 2))$ is given by
\begin{eqnarray} \label{chif}
\chi(\Gamma)^2+\chi(\Gamma) -\sum\limits_{v\in V(\Gamma)} (\mu(v)-1)(\mu(v)-2).
\end{eqnarray}
\end{corollary}

\begin{proof} It is easy to see that the number of vertices of $D=D(\Gamma, 2)$ is $V^2-V$ where $V=|V(\Gamma)|$.

Edges of $D$ are of the form either $ev$ or $ve$ where $v$ is a vertex
of $\Gamma$ and $e$ is an edge of $\Gamma$ not incident to $v$. The
number of edges of $\Gamma$ not incident to $v$ equals $E-\mu(v)$ where $E=|E(\Gamma)|$ is the number of edges of $\Gamma$. Hence the
total number of edges of $D$ is
$$2\cdot \sum_{v\in V(\Gamma)} (E-\mu(v))=2EV-4E= 2E(V-2).$$

The number of 2-dimensional cells of $D$ equals $$E^2-E- \sum_v \mu(v)(\mu(v)-1) = E^2+E -\sum_v \mu(v)^2.$$ Here $E^2$ is the number of all ordered pairs $ee'$ of edges and $E$ is the number of 2-cells of the form $ee$, while the last sum counts cells $ee'$ such that the intersection $e\cap e'$ is a single vertex.

Hence $\chi(D)$ equals
\begin{eqnarray*}
(V^2-V) -(2EV-4E) + \left(E^2+E - \sum_v \mu(v)^2\right) \\
= \, \chi(\Gamma)^2 + \chi(\Gamma) - \sum_v (\mu(v)-1)(\mu(v)-2).\end{eqnarray*}
\end{proof}

Note that Corollary \ref{euler} also follows from a more general theorem of Swiatkowski \cite{Sw} expressing the Euler characteristic of the configuration space $F(X, n)$ of an arbitrary polyhedron; see Corollary 2.7 in \cite{Far1}.

An important role in the subject is played by two well-known Kuratowski graphs $K_{5}$ and $K_{3,3}$.
For these graphs the configuration spaces $D(\Gamma, 2)$ are orientable surfaces of genus 6 and 4 respectively.
Moreover these two are the only graphs for which $D(\Gamma, 2)$ is a surface, see \cite{Ab}.
\begin{figure}[h]
\begin{center}
\resizebox{8cm}{4cm}{\includegraphics[56,336][575,582]{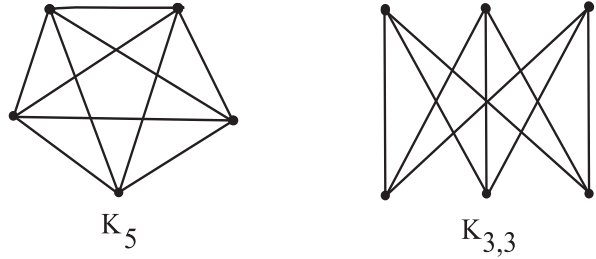}}
\end{center}
\caption{Graphs $K_5$ and $K_{3,3}$.} \label{kkk}
\end{figure}

We will also mention that $F(\Gamma, 2)$ and $D(\Gamma, 2)$ are path-connected assuming that $\Gamma$ is a finite graph which is not homeomorphic to the interval $[0,1]$, see Theorem 2 of \cite{Patty61}. Patty \cite{Patty61} also proved that spaces $F(\Gamma, 2)$ are aspherical, i.e. their homotopy groups $\pi_i(F(\Gamma, 2))$ vanish for $i\ge 2$. A more recent general result of Ghrist \cite{Gr} states that the space $F(\Gamma, n)$ is aspheric for any $n$ and for any finite graph $\Gamma$.

\begin{proposition}\label{epi}
Let $\Gamma$ be a connected finite graph which is not homeomorphic to the circle. Then the inclusion
$\alpha: F(\Gamma, 2) \to \Gamma\times \Gamma$
induces an epimorphism
\begin{eqnarray}
\alpha_\ast: H_1(F(\Gamma, 2)) \to H_1(\Gamma\times \Gamma).
\end{eqnarray}
\end{proposition}
\begin{proof} Any one-dimensional homology class of a topological space $X$ can be represented by a loop $S^1\to X$.
Hence Proposition \ref{epi} follows once we know that {\it any pair of continuous maps $\gamma, \gamma': S^1\to \Gamma$ can be changed by a continuous homotopy such that
for any point $z\in S^1$ one has $\gamma(z)\not= \gamma'(z)$}. Then $z\mapsto (\gamma(z), \gamma'(z))$ is a loop with values in $F(\Gamma, 2)$.

Since $H_1(\Gamma\times \Gamma)=H_1(\Gamma\times x_0) \oplus H_1(x_0\times \Gamma)$
where $x_0\in \Gamma$ is a base point and simple loops (i.e. loops without self intersections) generate $H_1(\Gamma)$,
it follows that it is enough to prove the statement of the previous paragraph assuming that one of the curves $\gamma,\gamma'$ is constant and the other is simple.

Let $\gamma: S^1\to \Gamma$ be a simple closed curve and $\gamma'$ a constant curve at a point $x_0\in \Gamma$. Our statement is trivial if
$x_0\not\in \gamma(S^1)$. In the case $x_0\in \gamma(\Gamma)$
we may find a point $x_0'\in \Gamma$ which does not belong to $\gamma(\Gamma)$ (here we use our assumption that $\Gamma$ is connected and is not homeomorphic to the circle). Deforming $\gamma'$ into the constant loop $\tilde \gamma':S^1\to \Gamma$ at $x_0'$ we obtain a deformation of the initial pair of loops to a pair of loops which never occupy the same location in the graph at the same time.
\end{proof}
\begin{corollary} For a
connected finite graph  $\Gamma$ which is not homeomorphic to $S^1$ the inclusion
$\alpha: F(\Gamma, 2) \to \Gamma\times \Gamma$
induces a monomorphism
\begin{eqnarray}
\alpha^\ast: H^1(\Gamma\times \Gamma) \to H^1(F(\Gamma, 2)).
\end{eqnarray}
\end{corollary}

\section{Intersection of cycles in a graph}

Let $\Gamma$ be a connected finite graph and let $d: \Gamma\times \Gamma\to \R$ be a metric on $\Gamma$ such that the length
of any edge of $\Gamma$ equals 1. We will also assume that the distance $d(x, y)$ between any two points $ x, y\in \Gamma$ equals the minimal length of
 a path connecting $x$ and $y$.

The complement $\Gamma\times \Gamma- D(\Gamma, 2)$ is an open neighbourhood of the diagonal $\Delta_\Gamma\subset \Gamma\times \Gamma$; we shall denote by $N$ its closure, i.e.
\begin{eqnarray}
N \, = \, N_\Gamma \, =\, \overline{\Gamma\times \Gamma- D(\Gamma, 2)}.\end{eqnarray}
Thus, a pair $(x,y)\in \Gamma\times \Gamma$ lies in $N$
if either $\supp\{x\}\cap \supp\{y\}$ is nonempty, or if at least
one of the points $x$ or $y$ is a vertex and $d(x,y)\leq 2$.

$N=N_\Gamma$ has an obvious cell structure. The 0-dimensional cells of $N$ are ordered
pairs $vw$ of vertices of $\Gamma$ such that $d(v,w)\leq 2$. The 1-dimensional
cells of $N$ are of the form $ev$ and $ve$ where $v$ is a vertex of
$\Gamma$, $e$ is an edge of $\Gamma$ and the distance between $v$ and $e$ is less than or equal to $1$.
The 2-cells of $N$ are of the
form $ee'$ where $e$ and $e'$ are edges of $\Gamma$ with $ e\cap
 {e'}\not= \emptyset$.

In the sequel the following set
\begin{eqnarray}
\partial N \, = \, \partial N_\Gamma = N\cap D(\Gamma, 2)
\end{eqnarray}
plays an important role.
This set is a one-dimensional cell complex (graph) having the following
cells. Zero-dimensional cells of $\partial N$ are of the form $vw$ where $v$
and $w$ are vertices of $\Gamma$ with $d(v,w)=1$ or $2$. One-dimensional cells of $\partial N$ are of two types: $ev$ (horizontal) and
$ve$ (vertical) where $v$ and $e$ are a vertex and an edge of $\Gamma$ and the
distance between $v$ and $ e$ equals 1.

Note that $\partial N$ can be viewed as the configuration space of two particles $x,
y\in\Gamma$ such that $d(x,y)$ lies between $1$ and $2$ and at least one of the
points $x$ and $y$ is a vertex of $\Gamma$.

Next we introduce {\it the intersection form}
\begin{eqnarray}\label{intersection}
I=I_\Gamma: H_1(\Gamma) \otimes H_1(\Gamma) \to H_2(N, \partial N).
\end{eqnarray}
This form measures intersection of cycles in $\Gamma$ and is similar in spirit to the classical intersection forms of cycles in manifolds.
The form (\ref{intersection}) is defined as follows. Consider the inclusion $j: \Gamma\times \Gamma \to (\Gamma\times \Gamma,D(\Gamma, 2))$ and the induced homomorphism $j_\ast$ on the two-dimensional homology. By the K\"unneth theorem $H_2(\Gamma\times \Gamma)$ can be identified with $H_1(\Gamma)\otimes H_1(\Gamma)$; besides, $H_2(\Gamma\times\Gamma, D(\Gamma, 2))$ can be identified with $H_2(N, \partial N)$ by excision.
After these identifications $j_\ast$ turns into homomorphism (\ref{intersection}).

We mention the following obvious properties of $I=I_\Gamma$:

\begin{lemma}\label{lm21} Suppose that two homology classes $z, z'\in H_1(\Gamma)$ can be realised by closed curves $\gamma, \gamma': S^1\to \Gamma$ such that
$\gamma(S^1) \cap \gamma'(S^1) = \emptyset$.
Then $I(z\otimes z')=0.$
\end{lemma}
A partial inverse to this statement is given later in Lemma \ref{inverse}.

\begin{lemma}\label{disjoint} For homology classes $z, z'\in H_1(\Gamma)$ one has
\begin{eqnarray}\label{taumin}
I(z'\otimes z) = - \tau_\ast(I(z\otimes z')),
\end{eqnarray}
where $\tau: (N, \partial N) \to (N, \partial N)$ denotes the canonical involution.
\end{lemma}

The minus sign appears in (\ref{taumin}) since $\tau_\ast (z\otimes z') = -z'\otimes z$.

The relevance of the intersection form $I$ to the problem of computing homology groups of $F(\Gamma, 2)$ follows from the following statement.


\begin{proposition}\label{prop6} Let $\Gamma$ be a finite connected graph which is not homeomorphic to the circle. Then
{\rm {(i)}} the group $H_2(F(\Gamma, 2))$ is isomorphic to the kernel of the intersection form
\begin{eqnarray}
H_2(F(\Gamma, 2)) \simeq \ker (I_\Gamma)
\end{eqnarray} and {\rm {(ii)}} the group $H_1(F(\Gamma, 2))$ is isomorphic to the direct sum
\begin{eqnarray}H_1(F(\Gamma, 2)) \simeq \coker(I_\Gamma) \oplus H_1(\Gamma) \oplus H_1(\Gamma).
\end{eqnarray}
 \end{proposition}
\begin{proof} Consider the homological sequence of the pair $(\Gamma\times \Gamma, F(\Gamma, 2))$.
If $\alpha$ denotes the embedding $F(\Gamma, 2)\to \Gamma \times \Gamma$ then the induced map $\alpha_\ast$ on one-dimensional homology is onto (by Proposition \ref{epi}). Moreover, one may use Lemma \ref{defret} and excision to identify $H_2(\Gamma\times \Gamma, D(\Gamma, 2))$ with $H_2(N, \partial N)$. This gives the following exact sequence
\begin{eqnarray}\label{secex}
\begin{array}{c}
 0\to H_2(F(\Gamma, 2)) \stackrel{\alpha_\ast}\to H_1(\Gamma) \otimes H_1(\Gamma) \stackrel {I_\Gamma}\to H_2(N, \partial N)\\ \\
 \stackrel \partial \to H_1(F(\Gamma, 2))\stackrel {\alpha_\ast} \to H_1(\Gamma\times\Gamma)\to 0.\end{array}\end{eqnarray}
 This exact sequence clearly implies statements (i) and (ii).
\end{proof}

First we mention the following simple but useful Corollary:

\begin{proposition}\label{prop24}
If $\Gamma$ is a connected finite graph not homeomorphic to $S^1$ then the inclusion $\alpha: F(\Gamma, 2)\to \Gamma\times \Gamma$ induces an epimorphism
\begin{eqnarray}
\alpha^\ast: H^2(\Gamma\times \Gamma) \to H^2(F(\Gamma, 2)).
\end{eqnarray}
\end{proposition}

This follows directly from the cohomological exact sequence of the pair $(\Gamma\times \Gamma, F(\Gamma, 2))$.

\begin{corollary}\label{ndn} For a connected graph $\Gamma$ which is homeomorphic to neither $S^1$ nor $[0,1]$ one has $H_i(N, \partial N)=0$ for all $i\not= 2$ and the group $H_2(N, \partial N)$ is free abelian of rank
\begin{eqnarray}\label{rkn}
 b_1(\Gamma) -1 + \sum_{v\in V(\Gamma)} \left(\mu(v)-1\right)\left(\mu(v)-2\right) .
\end{eqnarray}
\end{corollary}

\begin{proof} Proposition \ref{epi} implies that $$H_1(N, \partial N) = H_1(\Gamma\times \Gamma, F(\Gamma,2))=0$$ and obviously $H_0(N, \partial N)=0$.
The exact sequence (\ref{secex}) gives the equation
$$b_2(F) -b_1(\Gamma)^2 + {\rk}\,  H_2(N, \partial N) - b_1(F) + 2b_1(\Gamma)=0$$
where $F$ stands for $F(\Gamma, 2)$. This gives ${\rk} \, H_2(N, \partial N) = \chi(\Gamma)^2 - \chi(F).$ Now, taking into account (\ref{chif}) gives (\ref{rkn}).
The group $H_2(N, \partial N)$ is free abelian since $N$ has no 3-dimensional cells.
\end{proof}

Note that Corollary \ref{ndn} is false in the case when $\Gamma$ is homeomorphic to either $S^1$ or $[0,1]$.

\begin{corollary}
If $\Gamma$ is a tree then $H_2(F(\Gamma,2))=0$ and $$H_1(F(\Gamma, 2))\simeq  H_2(N, \partial N).$$
\end{corollary}
This follows directly from (\ref{secex}).

\begin{remark} {\rm The homology of $(N, \partial N)$ is independent of the graph subdivison and is a topological invariant of the graph $\Gamma$. Indeed,
$H_i(N, \partial N)\simeq H_i(\Gamma\times \Gamma, D(\Gamma, 2))\simeq H_i(\Gamma\times\Gamma, F(\Gamma, 2))$.
However the homotopy types of $N$ and $\partial N$ may depend on a particular triangulation of the graph $\Gamma$.
One may prove that {\it if $\Gamma$ is subdivided sufficiently fine such that each simple closed cycle passes through at least 5 edges then the projection on the first coordinate $N \to \Gamma$ is a homotopy equivalence}. We do not use this statement in this paper and leave it without proof.

Let $\Gamma$ be the triangle (graph of the letter $\Delta$). Then $N$ is homeomorphic to $S^1\times S^1$ and the projection $N\to \Gamma$ is not a homotopy equivalence. The projection $N\to \Gamma$ is not a homotopy equivalence also in the case when $\Gamma$ is the boundary of the square $\square$. These are two examples which should be excluded.}
\end{remark}

\section{Computing the intersection form}
First we describe an explicit recipe for computing the intersection form $I$.
Consider the cellular chain complex $C_\ast(N, \partial N)$ of the pair $(N, \partial N)$. Here $C_i(N, \partial N)$ is free abelian group generated by ordered pairs $aa'$ consisting of
closed oriented cells $a, a'$ of $\Gamma$ such that $a\cap a'\not=\emptyset$ and $\dim a + \dim a'=i$ where $i=0,1,2.$
Thus $C_2(N, \partial N)$ has as its basis the set of pairs $ee'$ of oriented edges of $\Gamma$ such that $e\cap e'\not=\emptyset$.
The group $C_1(N, \partial N)$ is freely generated by pairs $ve$ and $ev$ where $v$ is a vertex of $e$. The basis of the group $C_0(N, \partial N)$ is
the set of pairs $vv$ where $v\in V(\Gamma)$. The boundary homomorphism $\partial: C_i(N, \partial N) \to C_{i-1}(N, \partial N)$ acts as follows:
\begin{figure}[h]
\begin{center}
\resizebox{12cm}{4cm}{\includegraphics[14,320][578,545]{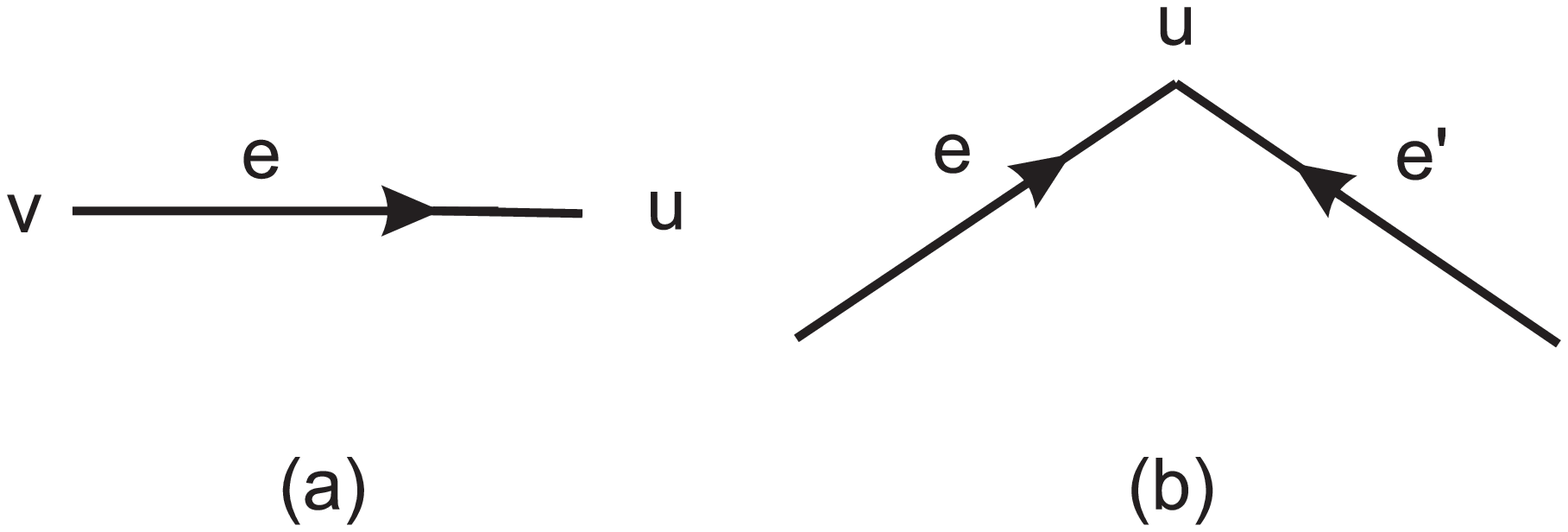}}
\end{center}
\end{figure}
$\partial (ee) = (u-v)e -e(u-v)$ and for $e\not= e'$ one has $\partial (ee') = ue'-eu$; besides $\partial (ue) =\partial (eu) =uu$, $\partial (ve)=\partial (ev)= -vv$ where relations between $e, e', u$ and $v$ are explained on the figure above.

The homology group $H_2(N, \partial N)$ coincides with the kernel of the boundary homomorphism
$\partial: C_2(N, \partial N) \to C_1(N, \partial N)$ and therefore we may view $H_2(N, \partial N)$ as being a subgroup of $C_2(N, \partial N)$.
Hence the intersection form $I$ might be thought of as taking values in the chain group $C_2(N, \partial N)$. Given two cycles $z=\sum n_i e_i$ and
$z'=\sum m_j e'_j$ their intersection $I(z\otimes z')$ equals
\begin{eqnarray}\label{formul}
I(z\otimes z') = \sum_{(i,j) \in A} n_im_j (e_ie'_j)\, \in \, C_2(N, \partial N)
\end{eqnarray}
where $A$ is the set of all pairs $(i,j)$ of indices such that $e_i\cap e'_j\not= \emptyset.$

\begin{lemma}\label{inverse}
For homology classes
$z, z'\in H_1(\Gamma)$ one has $I(z\otimes z')=0$ if and only if $z$ and $z'$ can be realised by cellular chains
$c=\sum n_ie_i$ and $c'=\sum m_je_j'$ which are disjoint, i.e. $e_i\cap e'_j=\emptyset$ for all $i, j$. Here $n_i, m_j \in \Z$ and $n_i\not=0$, $m_j\not=0$.
\end{lemma}
This lemma complements Lemma \ref{lm21}.

\section{Examples}

\subsection{Example: $\Gamma=K_5$} As the first example consider the case $\Gamma=K_5$. Vertices of $K_5$ will be denoted by the symbols $1, 2, 3, 4, 5$ and the edge connecting vertices $i$ and $j$ will be denoted by $(ij)$, where $i<j$. We assume that each such edge is oriented from $i$ to $j$. The union of all edges emanating from 5 forms a spanning tree. Hence a basis of the homology group $H_1(K_5)$ is formed by the cycles
$$C_{ij} = (ij) +(j5)-(i5), \quad i< j, \quad i,j=1, 2, 3, 4.$$
\begin{figure}[h]
\begin{center}
\resizebox{5cm}{4cm}{\includegraphics[143,361][419,575]{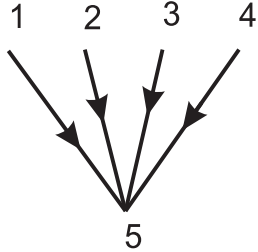}}
\end{center}
\caption{The maximal tree of $K_5$.} \label{k5}
\end{figure}
Computing their intersections using formula (\ref{formul}) we find
$$I(C_{12}\otimes C_{34}) = (25)(45) -(25)(35)-(15)(45) +(15)(35),$$

$$I(C_{13}\otimes C_{24}) = (35)(45) -(35)(25)-(15)(45) +(15)(25),$$

$$I(C_{14}\otimes C_{23}) = (45)(35) -(15)(35)-(45)(25) +(15)(25).$$
Continuing these calculations we obtain that the tensor $x\in H_1(K_5) \otimes H_1(K_5)$ given by
$$x=C_{12}\otimes C_{34} - C_{13}\otimes C_{24} + C_{14}\otimes C_{23}+ C_{34}\otimes C_{12} -C_{24}\otimes C_{13} +C_{23}\otimes C_{14}$$
satisfies $I(x)=0$. It
represents a nonzero and indivisible homology class in $H_2(F(K_5, 2))$, in view of exact sequence (\ref{secex}). Note that $x$ can be written in the form
\begin{eqnarray}
x= \sum_{(ijkl)} \epsilon_{(ijkl)} C_{ij}\otimes C_{kl},
\end{eqnarray}
where $(ijkl)$ runs over all permutations of indices $1,2,3, 4$ such that $i<j$ and $k<l$ and $ \epsilon_{(ijkl)}=\pm 1$ denotes the sign of the permutation.

Our notation $(ijkl)$ stands for the permutation $$(ijkl) = \left(
\begin{array}{cccc} 1& 2& 3& 4\\ i&j&k&l\end{array}\right).$$

We know that $F(K_5, 2)$ is homotopy equivalent to orientable surface of genus $6$ and hence $H_2(F(K_5, 2))\simeq \Z$.
In the case $\Gamma=K_5$ the groups appearing in exact sequence (\ref{secex}) have the following ranks: $\rk H_1(K_5)=6$, $\rk H_2(N, \partial N) = 35$, $\rk H_2(F(K_5, 2))= 1$ and $\rk H_1(F(K_5, 2))= 12$. Note that the intersection form $I: H_1(K_5)\otimes H_1(K_5) \to H_2(N, \partial N)$ is
 an epimorphism
in this case.

\subsection{Example: $\Gamma=K_{3,3}$} Consider now the case when $\Gamma=K_{3,3}$. We denote the vertices of the graph as shown on Figure \ref{k333}: upper vertices are labeled $a_1, a_2, a_3$ and lower vertices are $b_1, b_2, b_3$.
\begin{figure}[h]
\begin{center}
\resizebox{10cm}{4.6cm}{\includegraphics[60,533][542,777]{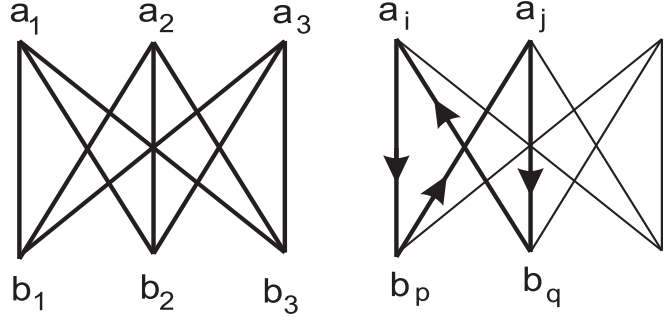}}
\end{center}
\caption{Graph $K_{3,3}$ and cycle $B^{ij}_{pq}$.} \label{k333}
\end{figure}
The edges of $K_{3,3}$ will be oriented from up to down; the edge starting at $a_i$ and ending at $b_p$ is denoted $(a_ib_p)$. For $i, j, p, q\in \{1, 2, 3\}$ with $i\not= j$ and $p\not=q$ consider the following cycle (and its homology class)
\begin{eqnarray}\label{cijpq}
B^{ij}_{pq} = (a_ib_p)-(a_jb_p) +(a_jb_q)-(a_ib_q)\in H_1(K_{3,3}).
\end{eqnarray}
Consider the tensor
$x\in H_1(K_{3,3})\otimes H_1(K_{3,3})$ given by the formula
\begin{eqnarray}\label{sum11}
x= \sum \epsilon_{(ijk)}\epsilon_{(pqr)}B^{ij}_{pq}\otimes B^{ik}_{pr}.
\end{eqnarray}
In this sum the symbols $(ijk)$ and $(pqr)$ run over all permutations of the indices $1,2, 3$ and $ \epsilon_{(ijk)}=\pm 1$ and $\epsilon_{(pqr)}=\pm 1$ denote the signs of these permutations.

We claim that {\rm {(i)}} $I(x)=0$ while {\rm {(ii)}} $x\not= 0$.
To prove (i) we note that
$I(B^{ij}_{pq}\otimes B^{ik}_{pr})$ (viewed as an element of $C_2(N, \partial N)$) equals
\begin{eqnarray*}\begin{array}{l}
(a_ib_p)(a_ib_p) - (a_ib_p)(a_kb_p) -(a_ib_p)(a_ib_r) - (a_jb_p)(a_ib_p) + \\
 (a_jb_p)(a_kb_p) -(a_ib_q)(a_ib_p) +(a_ib_q)(a_ib_r).
 \end{array}
\end{eqnarray*}
Each of these terms has the form $\pm (a_\alpha b_\beta)(a_\gamma b_\delta)$ where $\alpha, \beta, \gamma, \delta \in \{1, 2,  3\}$ and either $\alpha = \gamma$ or $\beta= \delta$.
In the case when $\alpha=\gamma$, permuting the remaining two indices $i, j, k \in \{1, 2, 3\}-\{\alpha\}$ we also obtain this term in the sum
$$I(x)\,  = \, \sum \epsilon_{(ijk)}\epsilon_{(pqr)}I(B^{ij}_{pq}\otimes B^{ik}_{pr})$$
but with the opposite sign. Similar arguments apply in all other cases.
Hence $I(x)=0$.

To prove (ii) we construct homomorphisms $f, g: H_1(K_{3,3})\to \Z$ such that $(f\otimes g)(x)\not=0$. Consider the maximal tree $T\subset K_{3,3}$ which is the union of all edges emanating from the vertices $a_3$ and $b_3$.
The remaining 4 edges $(a_1b_1)$, $(a_1b_2)$, $(a_2b_1)$, $(a_2b_2)$ label a basis of $H_1(K_{3,3})$. We denote by
$f: H_1(K_{3,3})\to \Z$ the homomorphism which equals 1 on the class represented by $(a_1b_2)$ and vanishes on the homology classes corresponding to 3
other edges. Explicitly, the value of $f$ on classes (\ref{cijpq}) is given by
\begin{eqnarray}
f(B^{ij}_{pq}) =\left\{
\begin{array}{ll}
1& \mbox{if $(i,p)=(1,2)$ or $(j,q)=(1,2)$},\\
-1 & \mbox{if $(j, p)=(1,2)$ or $(i,q)=(1,2)$},\\
0, & \mbox{otherwise}.
\end{array}
\right.
\end{eqnarray}
Similarly, define $g:H_1(K_{3,3})\to \Z$  to be the homomorphism which equals 1 on the class represented by $(a_2b_1)$ and vanishes on the homology classes corresponding to $(a_1b_1)$, $(a_1b_2)$, $(a_2b_2)$.
The value of $g$ on classes (\ref{cijpq}) is given by
\begin{eqnarray}
g(B^{ij}_{pq}) =\left\{
\begin{array}{ll}
1& \mbox{if $(i,p)=(2,1)$ or $(j,q)=(2,1)$},\\
-1 & \mbox{if $(j, p)=(2,1)$ or $(i,q)=(2,1)$},\\
0, & \mbox{otherwise}.
\end{array}
\right.
\end{eqnarray}

The number $(f\otimes g)(x)\in \Z$ can be represented in the form
\begin{eqnarray}\label{rhs}
(f\otimes g)(x) = \, \sum \epsilon_{(ijk)}\epsilon_{(pqr)} f(B^{ij}_{pq})  g(B^{ik}_{pr}) = \sum_{i, p=1}^ 3 A^i_p
\end{eqnarray}
where $A^i_p$ denotes the sum of terms appearing in (\ref{rhs}) with fixed indices $i$ and $p$.
For example,
$$A^1_1 = \sum \epsilon_{(1jk)}\epsilon_{(1q r)} f(B^{1j}_{1q})g(B^{1k}_{1r})$$
where $j,k,q,r=2,3$ and $j\not=k$, $q\not=r$. It it easy to see that $A_1^1$ contains only one nonzero term corresponding to
$j=3$, $k=2$,
$q=2$, $r=3$ and that $A^1_1=-1. $ Analyzing all 8 remaining possibilities one obtains that $A^i_p=-1$ for all $i,p=1,2,3$. Hence
$(f\otimes g)(x)=-9.$

We know that $F(K_{3,3}, 2)$ is homotopy equivalent to orientable surface of genus $4$ and hence $H_2(F(K_{3,3}, 2))\simeq \Z$.
 The groups appearing in exact sequence (\ref{secex}) have
in the case $\Gamma=K_{3,3}$
the following ranks: ${\rk} H_1(K_{3,3})=4$, $\rk H_2(N, \partial N) = 15$, $\rk H_2(F(K_{3,3}, 2))= 1$ and $\rk H_1(F(K_{3,3}, 2))= 8$. Note that the intersection form $I: H_1(K_{3,3})\otimes H_1(K_{3,3}) \to H_2(N, \partial N)$ is an epimorphism in this case as well.

\subsection{Discussion.} The two previous examples suggest that for all {\it \lq\lq well grown\rq\rq} graphs $\Gamma$ one may expect the intersection form $$I: H_1(\Gamma)\otimes H_1(\Gamma)\to H_2(N, \partial N)$$ to be an epimorphism or to have a small cokernel. If $I$ is surjective one has the following simple formulae
for the Betti numbers of the configuration space $F=F(\Gamma, 2)$:
\begin{eqnarray*}\begin{array}{l}
b_1(F) = 2b_1(\Gamma), \\ \\
b_2(F) =
 b_1(\Gamma)^2 -b_1(\Gamma) +1 - \sum_{v\in V(\Gamma)} \left(\mu(v)-1\right)\left(\mu(v)-2\right).
 \end{array}
\end{eqnarray*}
What are geometric conditions on the graph $\Gamma$ implying the surjectivity of the intersection form $I$? The case of planar graphs will be discussed in detail later; we will see that $I$ is never surjective for planar graphs however its cokernel has rank one under some quite general assumptions (see Theorem \ref{thm3}).

\section{Scalar intersection forms}

The homology of $(N, \partial N)$ can be computed using the cellular chain complex of $(N, \partial N)$.
In view of Corollary \ref{ndn} for $\Gamma$ not homeomorphic to $S^1$, $[0,1]$ one has the exact sequence
$$0\to C^0(N, \partial N) \to C^1(N, \partial N) \stackrel \delta\to C^2(N, \partial N) \to H^2(N, \partial N) \to 0$$
where $C^i(N, \partial N)$ is the dual of the free abelian group generated by the oriented cells of dimension $i$ of $N$ lying in $N-\partial N$.
Fix an orientation of each edge of $\Gamma$. Then $C^2(N, \partial N)$ can be viewed as the set of functions $f: E(\Gamma)\times E(\Gamma) \to \Z$ associating an integer to an ordered pair
$ee'$ of {\it oriented} edges of $\Gamma$ such that $e\cap e'\not=\emptyset$; here the case $e=e'$ is not excluded. Similarly, an element of
$C^1(N, \partial N)$ is a pair of functions
$$g: V(\Gamma)\times E(\Gamma) \to \Z, \quad h: E(\Gamma)\times V(\Gamma)  \to \Z,$$
such that $g(ve)$ and $h(ev)$ vanish assuming that $v\notin \partial e$. The coboundary map $\delta: C^1(N, \partial N) \to C^2(N,\partial N)$ is given by
the formula $$\delta(g,h)(ee')=g((\partial e) \cdot e') - h(e\cdot \partial e').$$

Fix a pair of oriented edges $ee'$ of $\Gamma$ with $e\cap e'\not=\emptyset$ and consider the cohomology class
\begin{eqnarray}\{f_{ee'}\} \in H^2(N, \partial N)\end{eqnarray} represented by the
delta-function cocycle  $f_{ee'}\in C^2(N, \partial N)$,
$$f_{ee'}(e_1e'_1) =\left\{
\begin{array}{ll}
1, & \mbox{if $e_1=e$ and $e_1'=e'$}, \\ \\
0, & \mbox{otherwise}.
\end{array}
\right.
$$
Note that the cohomology classes $\{f_{ee'}\}$ generate $H^2(N, \partial N)$ and there are some linear relations between them.

One may use cohomology classes $\{f_{ee'}\}$ to define the scalar intersection forms
\begin{eqnarray}
I_{ee'}: H_1(\Gamma) \otimes H_1(\Gamma) \to \Z.
\end{eqnarray}
\begin{definition}
For $z, z'\in H_1(\Gamma)$ we set
\begin{eqnarray}
I_{ee'}(z\otimes z') = \langle \{f_{ee'}\}, I(z\otimes z')\rangle \in \Z.
\end{eqnarray}
\end{definition}
In other words, $I_{ee'}(z\otimes z')$ is the evaluation of the cohomology class $\{f_{ee'}\}$ on the full intersection $I(z\otimes z')$. It is clear that
 the scalar intersection form can be explicitly computed as follows:
 \begin{lemma}\label{calc}
 Assume that homology classes $z, z'\in H_1(\Gamma)$ are presented as linear combinations
 $z=\sum n_i e_i$, and $z'=\sum m_j e_j$
 of distinct oriented edges of $\Gamma$. Then $I_{ee'}(z\otimes z') =n_i m_j$ where $e_i=e$ and $e_j=e'$.
 \end{lemma}
 Hence the intersection form $I_{ee'}$ counts instances when the first cycle $z$ passes along $e$ and the second cycle $z'$ passes along $e'$.
 The following Corollary follows either from Lemma \ref{calc} or from formula (\ref{taumin}).
 \begin{corollary} One has $I_{ee'}(z\otimes z') = I_{e'e}(z'\otimes z)$.
 \end{corollary}

 For future reference we also state:
 \begin{lemma} A tensor $x\in H_1(\Gamma)\otimes H_1(\Gamma)$ satisfies $$I(x)=0\in H_2(N, \partial N)$$ if and only if $I_{ee'}(x)=0$ for every pair of oriented edges $e$, $e'$ with $e\cap e'\not=\emptyset$.
 \end{lemma}

 This follows directly from the previous discussion.

\section{Planar graphs, I}

The following statement is one of the major results of this article.

\begin{theorem}\label{htwo} Let $\Gamma\subset \R^2$ be a planar graph and let $U_0, U_1, \dots, U_r$ be the connected components of the complement $\R^2-\Gamma$
with $U_0$ denoting the unbounded component. Then the second Betti number of $F(\Gamma, 2)$ equals the number of ordered pairs $(i,j)$ where $i, j\in \{0, 1, \dots, r\}$ are such that $$\bar U_i\cap \bar U_j=\emptyset.$$ For any such pair $(i,j)$ consider the torus $T_{ij}^2\subset F(\Gamma, 2)$ formed by the configurations where the first particle runs along the boundary of $U_i$ and the second particle runs along the boundary of $U_j$ respectively. The fundamental classes $[T^2_{ij}]\in H_2(F(\Gamma, 2))$ of these tori  freely generate $H_2(F(\Gamma, 2))$.
\end{theorem}

{\bf Remarks:}
(1) the tori $T^2_{ij}$ and $T^2_{ji}$ which appear in Theorem \ref{htwo} are disjoint and have to be counted separately. Hence, the second Betti number $b_2(F(\Gamma, 2))$ is even for any planar graph $\Gamma$.

(2) The involution $\tau: F(\Gamma, 2) \to F(\Gamma, 2)$ sends $T^2_{ij}$ onto $T^2_{ji}$. Hence, as a $\Z[\Z_2]$-module, $H_2(F(\Gamma, 2))$
is free of rank $\frac{1}{2}b_2(F(\Gamma, 2))$.

(3) We emphasize that in the statement of Theorem \ref{htwo} the indices $i, j$ can also take the value $0$.
\vskip 0.6cm

\begin{proof}[Proof of Theorem \ref{htwo}] Denote by $z_i\in H_1(\Gamma)$ the homology class of the cycle represented by the boundary of domain $U_i$, passed in the anti-clockwise direction, where $i=1, 2, \dots, r$. The classes $z_1, \dots, z_r$ form a free basis of $H_1(\Gamma)$. The class $z_0\in H_1(\Gamma)$ of the curve surrounding the graph, equals $z_1+ \dots+z_r$.

Suppose that $x\in H_1(\Gamma)\otimes H_1(\Gamma)$ is such that $I(x)=I_\Gamma(x)=0\in H_2(N, \partial N)$. Write
\begin{eqnarray}\label{x} x=\sum_{i, j=1}^r x_{ij}z_i\otimes z_j, \quad x_{ij}\in \Z.\end{eqnarray}
Our goal is to show that {\it $x$ can be uniquely expressed as a linear combination of tensors \begin{eqnarray}\label{one1}
\gamma_{ij} =z_i\otimes z_j, \quad\mbox{such that $i, j=1, \dots, r$ and $\bar U_i\cap \bar U_j=\emptyset$}
\end{eqnarray}
and also of tensors of the form
\begin{eqnarray}\label{two2}
\quad \quad \alpha_i =z_i \otimes z_0 = \sum_{j=1}^r z_i\otimes z_j  , \quad \mbox{and}\quad \beta_i = z_0\otimes z_i = \sum_{j=1}^r z_j\otimes z_i,\end{eqnarray}
such that $\bar U_i\cap \bar U_0=\emptyset$, where $i=1, \dots, r$.  }
The tensors (\ref{one1}) and (\ref{two2}) obviously
lie in the kernel of $I$. Theorem \ref{htwo} follows once the italicized claim has been proven.

One can rephrase this claim as follows:

{\it If a tensor $x\in H_1(\Gamma)\otimes H_1(\Gamma)$ represented in the form (\ref{x}) satisfies $I(x)=0$ then there exist unique integers $$a_1, a_2, \dots, a_r, \quad b_1, b_2, \dots, b_r\in \Z$$ (called left and right weights) such that
\begin{eqnarray}\label{main}
x_{ij}=a_i+b_j\end{eqnarray}
for any pair $(i,j)$ satisfying $\bar U_i \cap \bar U_j\not=\emptyset$;
moreover, one requires that
\begin{eqnarray}
a_i=0=b_i
\end{eqnarray}
for any $i=1, \dots, r$ satisfying $\bar U_i\cap \bar U_0\not=\emptyset$.
}

Indeed, if such weights $a_i, b_i$ are found then the linear combination
$$\sum_{i=1}^r a_iz_i\otimes z_0 + \sum_{j=1}^r b_jz_0\otimes z_j$$
has coefficient $x_{ij}$ in front of any tensor $z_i\otimes z_j$ with  $\bar U_i\cap \bar U_j\not=\emptyset$ and therefore it equals
$x$ minus a linear combination of tensors of type (\ref{one1}).


Note that it is enough to find the weight $a_i$ only since the other weights $b_i$ can be found from the relation $$x_{ii}=a_i+b_i.$$

Consider the following operation of {\it analytic continuation across an edge}.
\begin{figure}[h]
\begin{center}
\resizebox{5cm}{6cm}{\includegraphics[131,236][381,563]{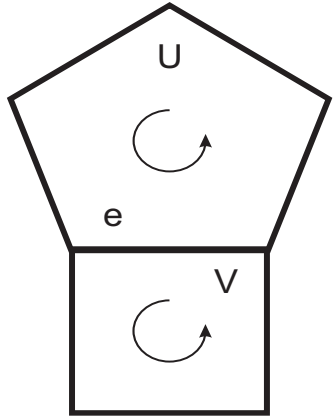}}
\end{center}
\caption{Two adjacent domains.}
\end{figure}
 Let $U$ and $V$ be two domains of the complement $\R^2-\Gamma$ having a common edge $e$.
Suppose that the weight $a_U$ is given. Then we have the following system of equations
\begin{eqnarray}\label{syst}
x_{UU} =a_U+b_U,\\
x_{UV} = a_U+b_V,\\
x_{VU} = a_V+b_U, \\
x_{VV} = a_V + b_V
\end{eqnarray}
to determine the remaining weights $b_U, a_V, b_V$. Here $x_{UU}$ , $x_{UV}$, $x_{VU}$ and $x_{VV}$ denote the corresponding coefficients of (\ref{x}). A solution to system (\ref{syst}) exists and the weight $a_V$ is given by
\begin{eqnarray}
a_V=  x_{VU}-x_{UU}+a_U\\
= x_{VV}-x_{UV}+a_U
\end{eqnarray}
assuming that the following compatibility
condition is satisfied
\begin{eqnarray}\label{edge}
x_{UU}+x_{VV} = x_{UV}+x_{VU}.
\end{eqnarray}
Note that this equation is indeed satisfied as follows by applying the intersection form $I_{ee}(x)$ where $e$ is the edge separating $U$ and $V$ and relying on Lemma \ref{calc}.

Hence, starting with an arbitrary value of the weight $a_U$ we may export it across an edge to a neighbouring face $V$.
This process may be continued inductively, along any sequence of faces and edges.

Two major questions arise:

1) Suppose that we perform this continuation process around a vertex $v$.
\begin{figure}[h]
\begin{center}
\resizebox{6cm}{4.3cm}{\includegraphics[30,349][463,665]{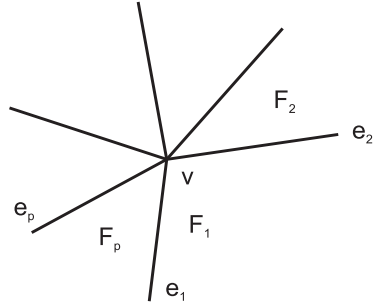}}
\end{center}
\caption{Exporting weights around a vertex.}
\end{figure}
We obtain a sequence of weights
$a_i, b_i$, where $i=1, \dots, p,$ such that $x_{ij}=a_i+b_j$
for all pairs satisfying $i=j \quad \mbox{or} \quad i-j=\pm 1.$
As compatibility conditions we have used all equations of the form $I_{ee}(x)=0$ for all edges $e$ separating the faces $U_i$.
Explicitly the solution is given by the formulae:
\begin{eqnarray}\label{explic}
a_j=\sum_{i=1}^{j-1}[x_{i+1,i}-x_{i,i}] + a_1,\\
b_j= x_{jj}- a_j,
\end{eqnarray}
where $j=1, \dots, p$.
 {\it Under which conditions one has
\begin{eqnarray}\label{three3}
x_{pq} =a_p+b_q\end{eqnarray}
for all remaining pairs $p, q$, i.e. for $p\not=q$ and $p-q\not=\pm 1$?}
Note that (\ref{three3}) is equivalent to
\begin{eqnarray}x_{qq}-x_{pq} = a_q-a_p\end{eqnarray}
which for $q>p$ in view of (\ref{explic}) is equivalent to
\begin{eqnarray}\label{four4}x_{qq}-x_{pq} = \sum_{i=p}^{q-1} [x_{i+1,i}-x_{i,i}]\end{eqnarray}
and for $q<p$ it can be written as
\begin{eqnarray}\label{five5}x_{pq} -x_{qq} = \sum_{i=q}^{p-1} [x_{i+1,i}-x_{i,i}].\end{eqnarray}

Consider two edges $e$ and $e'$ as shown on Figure \ref{around}, i.e. $e$ lies between $U_i$ and $U_{i+1}$ and $e'$ lies between $U_j$ and $U_{j+1}$.
\begin{figure}[h]
\begin{center}
\resizebox{7cm}{5cm}{\includegraphics[52,273][562,666]{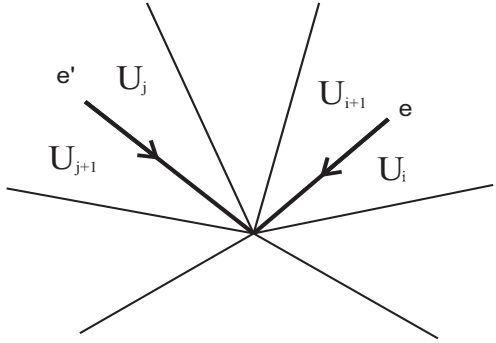}}
\end{center}
\caption{Domains around a vertex.}\label{around}
\end{figure}
Then the equation\footnote{Note that we do not require that the domains $U_i, U_{i+1}, U_j, U_{j+1}$ are distinct.}
$I_{ee'}(x)=0$
is equivalent to the equation
\begin{eqnarray}
x_{i,j} + x_{i+1, j+1} = x_{i, j+1} + x_{i+1, j}.
\end{eqnarray}
The latter equation can be rewritten as
\begin{eqnarray}
 x_{i+1, j+1} - x_{i, j+1} =  x_{i+1, j} - x_{i,j}.
\end{eqnarray}
It implies by induction that
$$x_{i+1, j}-x_{i, j}=x_{i+1, q}-x_{i,q}$$
for all $i, j, p$ and therefore (\ref{four4}) and (\ref{five5}) follow.

We conclude that there is no local monodromy, i.e. the result of the process of exporting weights around a vertex gives the initial weight and all obtained weights are compatible with each other.
The system of all obtained weights around a vertex is fully pairwise compatible, i.e. for any two domains $U_i$ and $U_j$ one has $x_{ij}=a_i+b_j$.

Suppose that we started at a domain $U$, fixed its weight $a_U$ arbitrarily, and continued it into some other face $V$ along a path of edges.
May the result depend on the path? The answer is negative. Indeed, weights of faces form a local system (flat line bundle) over the sphere with vertices of the graph removed. We know that
the monodromy around every vertex is trivial, but the loops surrounding vertices generate the fundamental group. Hence the whole monodromy is trivial.

Figure \ref{boundary} represents domains lying near the outer boundary of the graph.
\begin{figure}[h]
\begin{center}
\resizebox{7cm}{3.7cm}{\includegraphics[41,410][557,668]{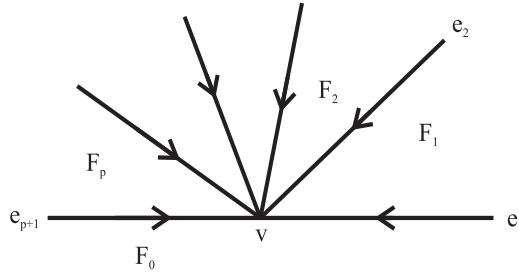}}
\end{center}
\caption{Planar domains near the outer component.}\label{boundary}
\end{figure}
The equation $I_{e_1e_1}(x)=0$ gives
$x_{11}=0$
and the equation $I_{e_ie_1}(x)=0$ (where $i=2, \dots, p$) gives
$x_{i-1,1} = x_{i,1}, \quad i=2, \dots, p.$
Hence we obtain that $$a_i=a_1=0\quad \mbox{for all}\quad i=2, \dots, p.$$

We may start our continuation process from a boundary domain; we may assume that the weight of this domain is trivial, $a_i=0$.
The argument above shows that moving along the boundary we will find that all other boundary domains have a trivial weights $a_j=0$.

This completes the proof.
\end{proof}

\begin{example} {\rm Consider the following graph $\Gamma= \Gamma_p$ consisting of two concentric circles and $p\ge 3$ radii.
\begin{figure}[h]
\begin{center}
\resizebox{6cm}{5cm}{\includegraphics[186,521][436,738]{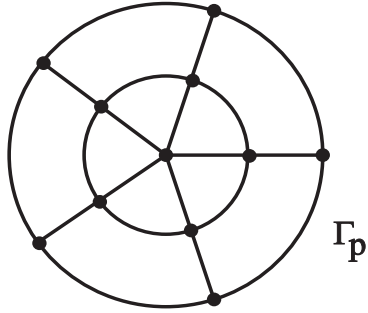}}
\end{center}
\caption{Graph $\Gamma_p$ for $p=5$.}
\end{figure}
We want to apply Theorem \ref{htwo}. The complement $\R^2-\Gamma$ consists of $2p+1$ domains $U_0, U_1, \dots, U_{2p}$ where $U_0$ denotes the exterior, $U_1, \dots, U_p$ are domains within the inner circle and $U_{p+1}, \dots, U_{2p}$ are domains of the annulus between the inner and outer circles. For any $i\in \{1, \dots, p\}$ one has $\bar U_i \cap \bar U_j=\emptyset$ for $j=0$ and for $p-3$ values $j\in \{p+1, \dots, 2p\}$.
Besides, for any $i\in \{p+1, \dots, 2p\}$ there exist $p-3$ values $j\in \{p+1, \dots, 2p\}$ such that $\bar U_i \cap \bar U_j=\emptyset$.
Applying Theorem \ref{htwo} we obtain $b_2(F(\Gamma, 2))= 3p^2 -7p$. Since $\chi(\Gamma) =1-2p$ and $\chi(F(\Gamma,2)) = 3p^2-11p$ (as follows from (\ref{chif})). This implies that $b_1(F(\Gamma, 2))= 4p+1= 2b_1(\Gamma)+1$. In view of exact sequence (\ref{secex}) it implies that for any $p\ge 3$ the cokernel of the intersection form
$I$ has rank one in this example.}\end{example}

\begin{example} {\rm Consider now a modification $\Gamma'_p$ of the previous example shown on Figure \ref{gammap1}.
\begin{figure}[h]
\begin{center}
\resizebox{5.5cm}{5cm}{\includegraphics[204,507][442,731]{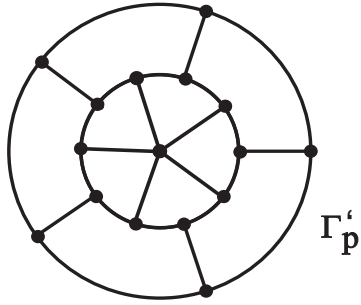}}
\end{center}
\caption{Graph $\Gamma'_p$ for $p=5$.}\label{gammap1}
\end{figure}
Here the picture inside the inner circle is rotated by the angle $\pi/p$. As above we denote by $U_0$ the outer domain, by $U_1, \dots, U_p$ the domains within the inner circle, and by $U_{p+1}, \dots, U_{2p}$ the domains lying in the annulus between the inner and outer circles. Each of the domains $U_1, \dots, U_p$ is disjoint from $U_0$ and from $p-2$ domains
$U_{p+1}, \dots, U_{2p}$. Besides, each $U_i$ with $i\in \{p+1, \dots, 2p\}$ is disjoint from $p-3$ domains $U_{p+1}, \dots, U_{2p}$.
Applying Theorem \ref{htwo} we find that $$b_2(F(\Gamma'_p,2))= 2p +(p-3)\cdot p + (p-2)\cdot (2p) = 3p^2 -5p.$$

Since $\chi(\Gamma'_p) = \chi(\Gamma_p) = 1-2p$ we may use (\ref{chif}) to find
$$\chi(F(\Gamma'_p)) = (1-2p)^2 +(1-2p) - 6p - (p-1)(p-2) = 3p^2-9p.$$
This gives $b_1(F(\Gamma'_p, 2)) = 4p+1$. Again, in view of exact sequence (\ref{secex}), we find that the cokernel of the intersection form $I$ has rank 1.}
\end{example}
\section{Planar graphs, II}

In this section we describe the first Betti number $b_1(F(\Gamma, 2))$ for a connected planar graph $\Gamma$.

\begin{proposition}\label{prop17}
For any connected planar graph $\Gamma\subset \R^2$ having an essential vertex the cokernel $\coker(I_\Gamma)$ of the intersection form
(\ref{intersection}) has rank $\ge 1$.
\end{proposition}
\begin{proof} We construct an explicit cohomology class \begin{eqnarray}\label{xi}
\xi\in H^2(N, \partial N)\end{eqnarray}
and show that (i) $\xi\not=0$ while (ii) the evaluation
$\langle \xi, I(z\otimes z')\rangle =0$ vanishes for any homology classes $z, z'\in H_1(\Gamma)$. Denote by
$$\psi: (\Gamma\times \Gamma, D(\Gamma, 2)) \to (\R^2, \R^2-\{0\})$$
the map given by $$\psi(x, y) = x-y.$$ Let $\xi=\psi^\ast(\iota)\in H^2(N, \partial N)$ be the image of the fundamental class $\iota \in H^2(\R^2, \R^2-\{0\})$
under the
induced map on cohomology
$$\psi^\ast: H^2(\R^2, \R^2-\{0\})\to H^2(\Gamma\times \Gamma, D(\Gamma, 2))\simeq H^2(N, \partial N).$$

To prove that $\xi$ is nonzero consider an essential vertex $u$ of $\Gamma$ and three edges, $e_1, e_2,e_3$, incident to it as shown on Figure \ref{triple1}.
\begin{figure}[h]
\begin{center}
\resizebox{4.4cm}{4cm}{\includegraphics[50,309][506,665]{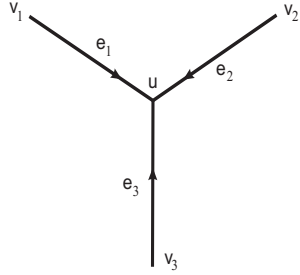}}
\end{center}
\caption{Three edges meeting at an essential vertex $u$.}\label{triple1}
\end{figure}
Consider the 2-dimensional chain $y\in C_2(N)$ given by
$$y=e_1(e_2-e_3) +e_2(e_3-e_1) +e_3(e_1-e_2).$$
It also can be represented in the form $$y=\sum_{(i,j,k)}\epsilon_{(ijk)} e_ie_j$$
where the sum is taken with respect to all permutations $(ijk)$ of $1,2,3$.
Clearly, $y$ has as its boundary the following 1-dimensional cycle
$$
\begin{array}{ccl}\partial y&=& v_1(e_3-e_2) + (e_1-e_3)v_2+ \\
&&v_3(e_2-e_1) + (e_3-e_2)v_1 +\\
&&v_2(e_1-e_3) + (e_2-e_1)v_3.\end{array}
$$
Here $v_1(e_3-e_2)$ is the motion of two particles such that the first point stands at $v_1$ and the second point moves from $v_3$ to $v_2$; the other parts of
$\partial y$ can be interpreted similarly.
It follows that $\partial y$ lies in $C_\ast(\partial N)$ and hence $y$ is a relative cycle. The evaluation $\langle \xi, \{y\}\rangle$ equals $\pm 1$ since the image of $\partial y$ under $\psi$ is a closed curve in the punctured plane $\R^2-\{0\}$ making one full twist around the origin. This claim is based on the observation that the angle which makes the ray from the first to the second point is always increasing.

Note that $\partial y$ can also be written in the following symmetric forms
\begin{eqnarray}\label{dy}
\partial y = -\sum_{(ijk)} \epsilon_{(ijk)}(v_ie_j+e_jv_i) = \sum_{(ijk)}\epsilon_{(ijk)}(e_iv_j - v_ie_j).
\end{eqnarray}

To prove (ii) consider two homology classes $z, z'\in H_1(\Gamma)$. Then $\langle \xi, I(z\otimes z')\rangle\in \Z$ equals the intersection number of cycles $z$ and $z'$ viewed as closed curves on the plane $\R^2$; it vanishes since $z$ and $z'$ bound on the plane.
\end{proof}

Next we present the result of Proposition \ref{prop17} in a different form.

Besides the natural embedding $\alpha: F(\Gamma, 2) \to \Gamma \times \Gamma$ (which appears in Propositions \ref{epi} and \ref{prop24}), the
configuration space $F(\Gamma, 2)$ embeds also into $F(\R^2, 2)$, the configuration space of two distinct points on the plane.

\begin{corollary}\label{corepi1}
For a connected planar graph $\Gamma\subset \R^2$ having an essential vertex, the map
$$\beta: F(\Gamma, 2) \to F(\R^2,2) \times \Gamma\times \Gamma$$ given by
$$(x, y) \mapsto ((x,y), x, y), \quad x, y \in \Gamma, \quad x\not=y$$ induces an epimorphism
\begin{eqnarray}\label{beta}
\beta_\ast: H_1(F(\Gamma, 2)) \to H_1(F(\R^2,2) \times \Gamma\times \Gamma)\end{eqnarray}
and a monomorphism
\begin{eqnarray*}\beta^\ast: H^1(F(\R^2,2)\times \Gamma\times \Gamma) \to H^1(F(\Gamma, 2)).\end{eqnarray*}
\end{corollary}

\begin{proof} Clearly, $F(\R^2,2)$ is homotopy equivalent to $S^1$ and therefore
$H_1(F(\R^2, 2)=\Z$.  In the proof of Proposition \ref{prop17} we constructed a loop $\partial y$ in $F(\Gamma, 2)$ such that the image of its homology class
$\{\partial y \} \in H_1(F(\Gamma, 2))$ under the map $\alpha_\ast: H_1(F(\Gamma, 2)) \to H_1(\Gamma\times \Gamma)$ vanishes and the image of the class $\{\partial y\}$ under the homomorphism $H_1(F(\Gamma, 2)) \to H_1(F(\R^2, 2))$
is a generator. Now Corollary \ref{corepi1} follows from Proposition \ref{epi}.
\end{proof}

\begin{theorem}\label{thm3} Let $\Gamma\subset \R^2$ be a connected planar graph such that every vertex $v$ has valence $\mu(v)\ge 3$. Denote by $U_0$, $U_1$,  $\dots, U_r$ the connected components of the
 complement $\R^2-\Gamma$ where $r=b_1(\Gamma)$ and $U_0$ is the unbounded component. Assume that:
 \begin{enumerate}
 \item[(a)] the closure of every domain $\bar U_i$ with $i= 1, \dots, r$ is contractible, and $\bar U_0$ is homotopy equivalent to the circle $S^1$;
\item[(b)] for every pair $i, j\in \{0, 1, \dots, r\}$
the intersection $\bar U_i \cap \bar U_j$ is connected.
\end{enumerate}
Then\footnote{Observe that the cokernel of the intersection form $I_\Gamma$ has rank one in this case, as follows by comparing the result of Theorem \ref{thm3} with Proposition \ref{prop6}.}
\begin{eqnarray}\label{btwo3}
b_1(F(\Gamma, 2)) = 2b_1(\Gamma) + 1
\end{eqnarray}
and $b_2(F(\Gamma, 2))$ equals
\begin{eqnarray}\label{btwo2}
b_1(\Gamma)^ 2- b_1(\Gamma) +2 - \sum_{v\in V(\Gamma)}(\mu(v)-1)(\mu(v)-2).
\end{eqnarray}
Here $V(\Gamma)$ denotes the set of vertices of $\Gamma$.
\end{theorem}

\begin{proof}
The number of all possible ordered pairs $(U_i, U_j)$ of distinct domains $i, j \in \{0, 1, \dots, r\}$ equals $r(r+1) =b_1(\Gamma)(b_1(\Gamma)+1)$. Our assumption implies that if $i\not=j$ and $\bar U_i \cap \bar U_j\not=\emptyset$ then the intersection $\bar U_i\cap \bar U_j$ is either a vertex or an edge. We say that a pair $(i,j)$ is of type one (type two) iff  $\bar U_i\cap \bar U_j$ is an edge (vertex, correspondingly).
Clearly, the number of pairs of type one
 equals $2E$ since each edge is incident to exactly two distinct domains $U_i$; here we use our assumptions (a) and (b) and $E=|E(\Gamma)|$ denotes the number of edges of $\Gamma$.
 The number of pairs $(i,j)$ of type two equals
 \begin{eqnarray}\label{sum-3}
 \sum_{v\in V(\Gamma)} \mu(v)\cdot (\mu(v)-3).
 \end{eqnarray}
 Indeed, consider a vertex $v$ and $\mu(v)$ domains incident to it. All these domains are distinct as follows from assumption (a).
 We observe that each of these domains $U_i$ forms a pair of type two with $\mu(v)-3$ of the domains $U_j$ incident to $v$. This explains formula (\ref{sum-3}).

Thus, applying Theorem \ref{htwo} we find
\begin{eqnarray}\label{sum-4}
\quad \quad b_2(F(\Gamma, 2)) = b_1(\Gamma)^2 + b_1(\Gamma) - 2E - \sum_{v\in V(\Gamma)} \mu(v)\cdot (\mu(v)-3).
\end{eqnarray}
By the Euler - Poincare theorem $V-E= 1-b_1(\Gamma)$; now formula (\ref{sum-4}) leads to (\ref{btwo2}), after some elementary transformations.

To prove (\ref{btwo3}) one writes $b_1(F)=1+b_2(F) - \chi(F),$ where $F=F(\Gamma, 2)$, and substitutes $b_2(F)$ and $\chi(F)$ using (\ref{btwo2}) and (\ref{chif}).

This completes the proof.
\end{proof}

Theorem \ref{thm3} and Corollary \ref{corepi1} imply the following result:

\begin{corollary}\label{cor74} For any planar graph $\Gamma\subset \R^2$ satisfying assumptions of Theorem \ref{thm3} the homomorphism (\ref{beta}) is an isomorphism\footnote{We do not know if the first homology group $H_1(F(\Gamma, 2))$ may have nontrivial torsion. Corollary \ref{cor74} holds with integral coefficients assuming that this torsion vanisihes.}
$$\beta_\ast: H_1(F(\Gamma, 2);\Q) \to H_1(F(\R^2, 2) \times \Gamma\times \Gamma;\Q).$$
\end{corollary}

\subsection*{Explicit generators of $H_1(F(\Gamma, 2);\Q)$} Next we describe a specific set of cycles whose homology classes form a free basis of $H_1(F(\Gamma, 2);\Q)$ assuming that $\Gamma$ satisfies conditions of Theorem \ref{thm3}. Let $U_0,U_1, \dots,U_r$ be the connected components of the complement $\R^2-\Gamma$ where $r=b_1(\Gamma)$ and $U_0$ denotes the unbounded component. For each $i=1, \dots, r$ let $c_i\in C_1(\Gamma)$ be the cellular chain representing the boundary $\partial U_i$ passed in the anticlockwise direction. Let $v_i$ be a vertex not incident to $c_i$. Then $c_iv_i$ and $v_ic_i$ are clearly cycles in $F(\Gamma, 2)$; these are $2r$ elements
of our basis.

To describe an additional basis element consider a triple of edges $e_\alpha,e_\beta, e_\gamma$ meeting at a vertex $u$ similar to the situation shown on Figure \ref{triple1}. Let $\partial e_\alpha=u-v_\alpha$, $\partial e_\beta=u-v_\beta$, $\partial e_\gamma=u-v_\gamma$, i.e. these edges meet at point $u$ and originate at $v_\alpha$, $v_\beta$ and $v_\gamma$ correspondingly. The formula
$$\{e_\alpha, e_\beta, e_\gamma\} \, = \, \sum_{(ijk)}\epsilon_{(ijk)}(v_ie_j+e_jv_i)$$ (compare (\ref{dy}); here $(ijk)$ runs over all permutations of indices
$\alpha,\beta, \gamma$) gives a cycle in $F(\Gamma, 2)$ and its homology class together with the classes
$\{c_iv_i\}$, $\{v_ic_i\}$ (see the previous paragraph) form a free basis of the group $H_1(F(\Gamma, 2))$. This follows from the arguments of the proof of Proposition \ref{prop17}.

\begin{figure}[h]
\begin{center}
\resizebox{8cm}{4cm}{\includegraphics[15,308][509,592]{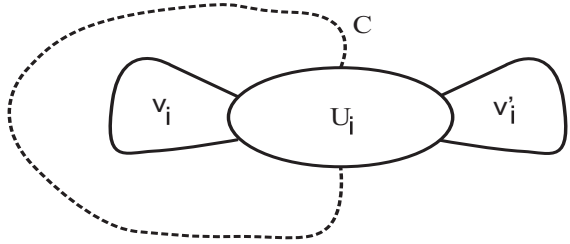}}
\end{center}
\caption{A separating domain.}\label{separating}
\end{figure}
Note that under assumptions of Theorem \ref{thm3} the homology classes of the cycles $c_iv_i$ and $v_ic_i$ are independent of the choice of the points
$v_i\in \Gamma - \partial U_i$, where $i=1, \dots, r$.
This follows from the observation that
 the complement $\Gamma-\partial U_i$ is path connected. Indeed if two points
$v_i, v'_i\in \Gamma$ lie in different connected components of $\Gamma-\partial U_i$ then there exists an arc $C\subset U_0$ with $\partial C=C\cap \partial U_0$ and such that the points $v_i$ and $v'_i$ belong to different connected components of $\R^2- (C\cup \partial U_i)$. This implies that the intersection $\bar U_i\cap \bar U_0$ is disconnected, contradicting our assumptions, see Figure \ref{separating}..

\begin{example} {\rm Consider graphs $\Gamma_1$ and $\Gamma_2$ shown in Figure \ref{gamma12}.

Graph $\Gamma_1$ does not satisfy condition (b) of Theorem \ref{thm3} since the intersection $\bar U_1 \cap \bar U_2$ is disconnected.
We find that $b_2(F(\Gamma_1, 2))=2$, $\chi(\Gamma) = -2$, $b_1(\Gamma_1) =3$, $\chi(F(\Gamma_1,2)) = -6$, and hence
$$b_1(F(\Gamma_1, 2)) = b_2(F(\Gamma_1, 2)) +1 - \chi(F(\Gamma_1, 2)) = 9.$$
We see that the conclusion of Theorem \ref{thm3} is false in this case.\begin{figure}[h]
\begin{center}
\resizebox{6cm}{4cm}{\includegraphics[50,192][534,570]{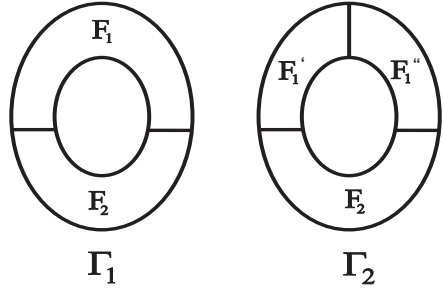}}
\end{center}
\caption{Graphs $\Gamma_1$ and $\Gamma_2$.}\label{gamma12}
\end{figure}

Graph $\Gamma_2$ is obtained from $\Gamma_1$ by dividing $U_1$ into two domains $U_1'$ and $U_1''$. Graph $\Gamma_2$ satisfies conditions of Theorem \ref{thm3}.
We obtain $b_2(F(\Gamma_2, 2)) =2$, $\chi(\Gamma_2)= -3$, $b_1(\Gamma_2) = 4$, $\chi(F(\Gamma_2, 2)) = -6$ and
$$b_1(F(\Gamma_2, 2)) = b_2(F(\Gamma_2, 2)) +1 - \chi(F(\Gamma_2, 2)) = 9=2b_1(\Gamma_2)+1.$$}
\end{example}

\begin{example} {\rm Consider the graph $\Gamma$ shown in Figure \ref{gamma15}. Clearly it does not satisfy condition (a) of Theorem \ref{thm3} as the closures of two of the domains of the complement are not simply connected. Let us show that the conclusion of Theorem \ref{thm3} is false in this case.
\begin{figure}[h]
\begin{center}
\resizebox{8cm}{5cm}{\includegraphics[37,329][524,665]{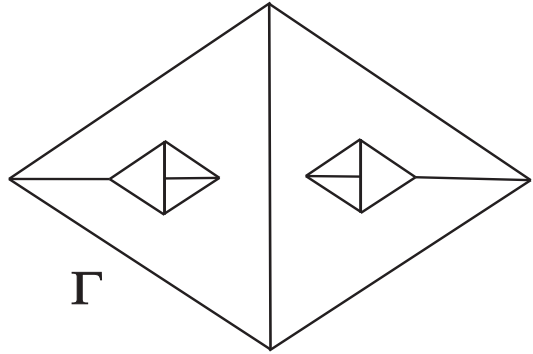}}
\end{center}
\caption{Graph not satisfying Theorem \ref{thm3}, condition (a).}\label{gamma15}
\end{figure}

We find in this example $V=14$, $E=21$, $\chi(\Gamma)= -7$ and $b_1(\Gamma)=8$. Computing $\chi(F(\Gamma, 2))$ via formula (\ref{chif}) gives $\chi(F(\Gamma, 2))= 14$. Counting pairs of disjoint domains gives (by Theorem \ref{htwo}) $b_2(F(\Gamma, 2))= 42$. Hence we find that $b_1(F(\Gamma, 2)) = 1+b_2(F(\Gamma, 2)) -\chi(F(\Gamma, 2)) = 29 \not= 17 = 2b_1(\Gamma)+1.$}
\end{example}

\begin{example} {\rm Consider now the following modification of the above graph obtained by splitting two domains, see Figure \ref{gamma15+}.
\begin{figure}[h]
\begin{center}
\resizebox{8cm}{5cm}{\includegraphics[37,329][524,665]{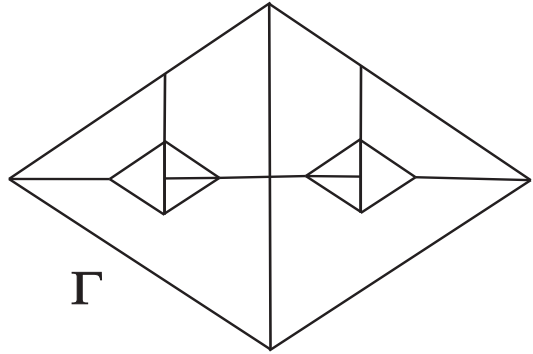}}
\end{center}
\caption{Graph satisfying assumptions of Theorem \ref{thm3}.}\label{gamma15+}
\end{figure}

In this example $V=17$, $E=28$, $\chi(\Gamma)= -11$ and $b_1(\Gamma)=12$. Computing $\chi(F(\Gamma, 2))$ via formula (\ref{chif}) gives $\chi(F(\Gamma, 2))= 56$. Counting pairs of disjoint domains gives (by Theorem \ref{htwo}) $b_2(F(\Gamma, 2))= 80$. Hence we find that $b_1(F(\Gamma, 2)) = 1+b_2(F(\Gamma, 2)) -\chi(F(\Gamma, 2)) = 25 = 2b_1(\Gamma)+1.$}
\end{example}

\section{The cup-product}

In this section we study the cup-product
\begin{eqnarray}
\qquad \cup: H^1(F; \Q) \times H^1(F; \Q) \to H^2(F; \Q), \, \mbox{where}\, F=F(\Gamma, 2).
\end{eqnarray}
Here $\Gamma\subset \R^2$ is a connected planar graph having an essential vertex.

Let $U_1, \dots, U_r$ denote the bounded connected components of the complement $\R^2-\Gamma$. Here $r=b_1(\Gamma)$ is the first Betti number of $\Gamma$.
Let $U_0$ denote the unbounded component of $\R^2-\Gamma$.
The boundary cycle of $U_i$ oriented anticlockwise is denoted by $z_i\in H_1(\Gamma)$, where $i=0, 1, \dots, r$. The homology classes $z_1, \dots, z_r$ form a basis of $H_1(\Gamma)$ and $z_0= z_1 + \dots+z_r$.

Denote
$$J(\Gamma) \, =\, \{(i, j); \, \bar U_i \cap \bar U_j =\emptyset,\,  \, i, j= 0, 1, \dots, r\}.$$
For $(i,j)\in J(\Gamma)$ denote by $T^2_{ij}\subset F(\Gamma, 2)$ the torus representing the set of all configurations when the first particle runs along the boundary of $U_i$ and the second particle runs along the boundary of $U_j$. We orient $\partial U_i$ and $\partial U_j$ in the anti-clockwise direction; then the torus $T^2_{ij}$ is naturally oriented.
By Theorem \ref{htwo} the homology classes of these tori
$$[T^2_{ij}] \in H_2(F; \Q), \quad (i,j) \in J(\Gamma)$$
form a basis of the vector space $H_2(F; \Q)$.

Let
$$\eta_{ij} \in H^2(F; \Q), \quad (i,j) \in J(\Gamma)$$
be the dual basis of cohomology classes. Hence,
$$\langle \eta_{ij}, [T^2_{kl}]\rangle = \left\{
\begin{array}{ll}
1, & \mbox{if} \quad (i,j)=(k,l),\\

0, & \mbox{otherwise}.
\end{array}
\right.$$

First we describe the cup-product of classes lying in the image of the homomorphism
$$\alpha^\ast: H^1(\Gamma\times \Gamma; \Q) \to H^1(F; \Q)$$
induced by the inclusion $\alpha: F\to \Gamma\times \Gamma$. Recall that by Proposition \ref{epi} $\alpha^\ast$ is injective assuming that $\Gamma$ is not homeomorphic to $S^1$.

\begin{theorem}\label{thm81}
Given cohomology classes $\xi^\pm, \eta^\pm \in H^1(\Gamma; \Q)$ consider the classes $\xi, \eta\in H^1(F; \Q)$ defined by the formulae
$$\xi=\alpha^\ast(\xi^+\times 1 +1\times \xi^-), \quad \eta=\alpha^\ast(\eta^+\times 1 +1\times \eta^-),$$
where $F=F(\Gamma, 2)$.
Their cup-product $\xi\cup \eta \in H^2(F; \Q)$ is given by
\begin{eqnarray}\label{cup}\xi\cup \eta = \sum_{(i,j)\in J(\Gamma)}\left[
\langle \eta^+, z_i\rangle \langle \xi^-, z_j\rangle
- \langle \xi^+, z_i\rangle \langle \eta^-, z_j\rangle
\right]\cdot \eta_{ij}.\end{eqnarray}
\end{theorem}

\begin{proof}
Firstly, one has
\begin{eqnarray*}\xi\cup \eta&=& \alpha^\ast((\xi^+\times 1 +1\times \xi^-)\cup(\eta^+\times 1 +1\times \eta^-))\\
&=&\alpha^\ast(\xi^+\times \eta^--\eta^+\times \xi^-).
\end{eqnarray*}
Secondly, evaluating the cup-product $\xi\cup \eta$ on a homology class $[T^2_{ij}]\in H_2(F(\Gamma, 2); \Q)$ for some $(i,j)\in J(\Gamma)$ we find
\begin{eqnarray*}
\langle \xi\cup \eta,[T^2_{ij}]\rangle  &=& \langle \alpha^\ast(\xi^+\times \eta^--\eta^+\times \xi^-), [T^2_{ij}]\rangle\\
&=& \langle(\xi^+\times \eta^- -\eta^+\times \xi^-),  \alpha_\ast[T^2_{ij}] \rangle \\
&=& \langle(\xi^+\times \eta^- - \eta^+\times \xi^-),  z_i \times z_j\rangle \\
&=&  - \langle \xi^+, z_i\rangle \langle \eta^-, z_j\rangle   + \langle \eta^+, z_i\rangle \langle \xi^-, z_j\rangle.
\end{eqnarray*}
The minus sign is a consequence of Proposition 7.14 from Chapter 7 of \cite{D}.
This proves formula (\ref{cup}).
\end{proof}

Formula (\ref{cup}) can also be presented in the following form.

Let $u_1,  \dots, u_r\in H^1(\Gamma; \Q)$ be the basis dual to $z_1, \dots, z_r\in H_1(\Gamma; \Q)$.
Denote
$$\xi_i=\alpha^\ast(u_i\times 1), \quad \eta_i =\alpha^\ast(1\times u_i) \in H^1(F(\Gamma, 2); \Q), \quad i=1, \dots, r.$$
Then
\begin{eqnarray}
\xi_i \cup \xi_j=0=\eta_i\cup \eta_j\quad \mbox{for all} \quad i,j=1, \dots,r
\end{eqnarray}
and
\begin{eqnarray}\label{cup1}
\xi_i\cup \eta_j = - \epsilon_{ij} \eta_{ij} - \epsilon_{i0} \eta_{i0} -\epsilon_{0j} \eta_{0j} \, \in H^2(F(\Gamma, 2); \Q),
\end{eqnarray}
where $\epsilon_{ij}$ denotes
$$\epsilon_{ij} = \left\{
\begin{array}{ll}
1, & \mbox{if} \, \, (i,j) \in J(\Gamma),\\ \\
0, & \mbox{if} \, \, (i,j) \notin J(\Gamma).
\end{array}
\right.
$$
To prove (\ref{cup1}) we observe that
$$\xi_i\cup \eta_j =- \sum_{(k,l)\in J(\Gamma)} \langle u_i , z_k\rangle\langle u_j, z_l\rangle\cdot \eta_{kl}$$
as follows from (\ref{cup}). In this sum only three terms might be nonzero; they correspond to cases $(k,l)=(i, j)$, $(k,l)=(i,0)$ or $(k,l)=(0,j)$; each of these cases happens iff the corresponding pair lies in $J(\Gamma)$.

\begin{definition}
Let $c\in C_1(\Gamma)$ be a cycle and $v\in \Gamma$ be a vertex not incident to edges which appear in $c$ with nonzero coefficients. Then $vc$ and $cv$ are cycles in $F(\Gamma, 2)$. We will say that a cohomology class $\xi\in H^1(F(\Gamma, 2); \Q)$ is special if the evaluation $\langle \xi, vc\rangle =0=\langle \xi, cv\rangle$ vanishes for any pair $c$ and $v$ as above.
\end{definition}

\begin{theorem}\label{thmspecial} Let $\Gamma$ be a planar graph. Then for any special cohomo\-logy class $\xi\in H^1(F(\Gamma, 2); \Q)$
one has
$\xi\cup \eta =0$ for any class $\eta\in H^1(F(\Gamma, 2); \Q)$
\end{theorem}
\begin{proof} For any pair $(i,j)\in J(\Gamma)$ consider the torus $T^2_{ij}\subset F(\Gamma, 2)$. Given $\xi, \eta\in H^1(F(\Gamma, 2); \Q)$ as above consider the restrictions $\xi'=\xi|T^2_{ij}$ and $\eta'=\eta|T^2_{ij}$, where $\xi', \eta' \in H^1(T^2_{ij}; \Q)$. Then
$$\langle \xi\cup \eta, [T^2_{ij}]\rangle = \langle \xi'\cup \eta', s_{ij}\rangle$$
with $s_{ij}\in H_2(T^2_{ij}; \Q)$ denoting the fundamental class of the torus $T^2_{ij}$. Hence Theorem \ref{thmspecial} follows once we show that $\xi'=0$ for any special
cohomology class $\xi$.

Choose points $v_i\in \partial U_i$ and $v_j\in \partial U_j$. Since $\bar U_i$ and  $\bar U_j$ are disjoint, the cycles $v_i(\partial U_j)$ and $(\partial U_i )v_j$ lie in $F(\Gamma, 2)$ and  $\xi$  evaluates trivially on these cycles (as $\xi$ is special); but these cycles generate $H_1(T^2_{ij}; \Q)$ implying $\xi'=0$.
\end{proof}

\begin{theorem}\label{thmlast}
 Let $\Gamma\subset \R^2$ be a connected planar graph such that every vertex $v$ has valence $\mu(v)\ge 3$. Denote by $U_0$, $U_1$,  $\dots, U_r$ the connected components of the
 complement $\R^2-\Gamma$ where $r=b_1(\Gamma)$ and $U_0$ is the unbounded component. Assume that:
 \begin{enumerate}
 \item[(a)] the closure of every domain $\bar U_i$ with $i= 1, \dots, r$ is contractible, and $\bar U_0$ is homotopy equivalent to the circle $S^1$;
\item[(b)] for every pair $i, j\in \{0, 1, \dots, r\}$
the intersection $\bar U_i \cap \bar U_j$ is connected.
\end{enumerate}
Then there exists a nonzero special cohomology class $$\eta\in H^1(F(\Gamma, 2); \Q),$$ defined uniquely up to sign, such that any class $\xi\in H^1(F(\Gamma, 2); \Q)$ can be uniquely represented in the form
\begin{eqnarray}\label{special}\xi = \alpha^\ast(u^+\times 1 +1\times u^-) +\lambda\eta\end{eqnarray}
where $u^\pm \in H^1(\Gamma; \Q)$ and $\lambda \in \Q$.
\end{theorem}
\begin{proof}
In the discussion after Corollary \ref{cor74} we constructed a specific basis $z_1, \dots, z_{2r+1}\in H_1(F(\Gamma, 2); \Q)$ where $r=b_1(\Gamma)$.
The classes $z_1, \dots, z_{2r}$ are represented by closed curves of the form $c_iv_i$ and $v_ic_i$ with $c_i$ denoting the boundary of $U_i$ oriented
in the anticlockwise direction and $v_i\in \Gamma - \bar U_i$. The remaining class $z_{2r+1}$ is determined uniquely up to a sign. Consider the dual basis
$z^\ast_i\in H^1(F(\Gamma, 2); \Q)$, $i=1, \dots, 2r+1$. Then the classes $z^\ast_1, z^\ast_2, \dots, z^\ast_{2r}$ generate the image
of the homomorphism $\alpha^\ast: H^1(\Gamma \times \Gamma; \Q) \to H^1(F(\Gamma, 2); \Q)$ and the class $z^\ast_{2r+1}$ is special. This implies Theorem
\ref{thmlast}.
\end{proof}

Theorems \ref{thm81}, \ref{thmspecial} and \ref{thmlast} fully describe the structure of the cohomology algebra $H^\ast(F(\Gamma, 2); \Q)$.
\bibliographystyle{amsalpha}

\end{document}